\documentclass[11pt]{amsart}

\usepackage[T1]{fontenc}
\usepackage[latin1]{inputenc}

\usepackage[a4paper]{geometry}

\usepackage{amssymb}
\usepackage{amsmath}
\usepackage{amsthm}
\usepackage[all, 2cell, poly]{xypic}
\SelectTips{cm}{10}

\usepackage{version}
\usepackage{mathrsfs}
\usepackage{mathtools}

\usepackage{enumerate}

\include{macros}

\author{Dimitri Ara}

\address{Dimitri Ara, Radboud Universiteit Nijmegen, Institute for
Mathematics, Astrophysics, and Particle Physics, Heyendaalseweg 135, 6525 AJ
Nijmegen, The Netherlands}
\email{d.ara@math.ru.nl}
\urladdr{http://www.math.ru.nl/~dara}



\title{Strict $\infty$-groupoids are Grothendieck \oo-groupoids}

\begin{document}

\begin{abstract}
We show that there exists a canonical functor from the category of strict
\oo-groupoids to the category of Grothendieck \oo-groupoids and that this
functor is fully faithful. As a main ingredient, we prove that free
strict \oo-groupoids on a globular pasting scheme are weakly contractible.
\end{abstract}

\maketitle

\section*{Introduction}

The purpose of this paper is to prove that strict \oo-groupoids are
Grothendieck \oo-groupoids or, more precisely, that there exists a canonical
fully faithful functor from the category of strict \oo-groupoids to the
category of Grothendieck \oo-groupoids (the morphisms of Grothendieck
\oo-groupoids we are considering in this article are the strict ones).

The notion of Grothendieck \oo-groupoid was introduced by Grothendieck in
his famous letter to Quillen (this letter constitutes the first thirteen
sections of \cite{GrothPS}). Roughly speaking, a Grothendieck \oo-groupoid
is an \oo-graph (or globular set) endowed with operations similar to the one
of strict \oo-groupoids, with coherences (making it a strict \oo-groupoid up
to these coherences), with coherences between these coherences, and so on.
Grothendieck explained in the letter how to construct an \oo-groupoid
$\Pi_\infty(X)$  out of a topological space $X$ and he conjectured that this
\oo-groupoid $\Pi_\infty(X)$ (up to some notion of weak equivalence)
classifies the homotopy type of $X$.  This conjecture is now often referred
to as the \emph{homotopy hypothesis}. A precise statement of the conjecture
is given at the very end of \cite{AraHomGr}. (It is well-known that strict
\oo-groupoids are not sufficient for this purpose. See for instance
\cite[Example 6.7]{BrownHigginsClass}, \cite[Chapter~4]{SimpsonHTHC} or
\cite{AraTypHomStr}.)

To define precisely Grothendieck \oo-groupoids, one has to give a
description of these higher coherences. It seems hopeless to describe them
explicitly. Even for $3$\nbd-groupoids, the explicit description of the
coherences (see~\cite{GPSTricat}) is very involved. Grothendieck's main
insight is that one can generate all the coherences by induction in a simple
way. An intuitive explanation of this inductive description is given in the
introductions of~\cite{AraThesis} and~\cite{AraHomGr}.

The definition of Grothendieck \oo-groupoids we use in this paper is not
exactly the original one (this original definition is explained in
\cite{MaltsiGr}). First, we use the simplification introduced by
Maltsiniotis in \cite{MaltsiCat} and \cite{MaltsiGrCat}. Second, we use the
slight modification we introduced in \cite{AraHomGr}. The purpose of this
modification was precisely to make canonical the inclusion functor from
strict \oo-groupoids to Grothendieck \oo-groupoids as announced in Remark
2.8.1 of loc.~cit.

\smallskip

Let us come back to the main statement of the paper. Consider for the moment
the following weaker statement: every strict \oo-groupoid can be endowed
with the structure of a Grothendieck \oo-groupoid. This statement might seem
tautological at first: it says, in some sense, that strict \oo-groupoids are
weak \oo-groupoids. This apparent paradox comes from the fact that the
coherences which are part of the algebraic structure of Grothendieck
\oo-groupoid are not explicitly defined (in the sense that they are
generated by induction) and it is thus not clear that strict \oo-groupoids
admit such coherences. Of course, if it was not the case, the notion of
Grothendieck \oo-groupoid would have to be corrected.

It turns out that this weaker statement is equivalent to the following
statement: free strict \oo-groupoids on a globular pasting scheme are
weakly contractible.

To prove the weak contractibility of these strict \oo-groupoids, we use the
following strategy. First, we reduce to the case of the free strict
\oo-groupoid on the $n$-disk by using the Brown-Golasi\'nski model category
structure on strict \oo-groupoids (see~\cite{BrownGolas}
and~\cite{AraMetBrown}). The case of the $n$-disk is then proved using the
path object of cylinders introduced by Métayer in \cite{MetPolRes} and
studied in details in \cite{LMW}. The idea of using these cylinders was
suggested to us by Yves Lafont.

This Grothendieck \oo-groupoid structure on a strict \oo-groupoid is not
unique if we stick to the Grothendieck's original definition. We show that,
if we use the modified definition we introduced in \cite{AraHomGr}, this
structure becomes unique. This shows that there exists a canonical functor
from the category of strict \oo-groupoids to the category of Grothendieck
\oo-groupoids, as announced. Finally, we show that this functor is fully
faithful.

\smallskip

Our paper is organized as follows. In the first section, we recall basic
definitions and facts about strict \oo-groupoids and their weak
equivalences. In the second section, we recall the definition of
Grothendieck \oo-groupoids and we state our main result. In particular, we
introduce the fundamental notions of globular extension, contractible
extension, globular presheaf and coherator. In the third section, we
introduce the globular extension~$\Thtld$, making the link between the world
of strict \oo-groupoids and the world of Grothendieck \oo-groupoids. We
explain the relation between the properties of $\Thtld$ and our main result.
The fourth section is dedicated to the homotopy theory of strict
\oo-groupoids.  In particular, we introduce the Brown-Golasi\'nski model
structure and Métayer's path object of cylinders.  The fifth section is the
technical heart of the article. We show that the free strict \oo-groupoid on
the $n$-disk (seen as an \oo-graph) is weakly contractible. To do so, we
exhibit an $n$-cylinder leading to a homotopy between the identity functor
of this \oo-groupoid and a constant functor. In the sixth section, we prove
that the globular extension $\Thtld$ is canonically contractible. This means
in particular that free strict \oo-groupoids on a globular pasting scheme
are weakly contractible. In the seventh section, we study conditions on a
globular extension to get a fully faithful functor from strict \oo-groupoids
to globular presheaves on this globular extension. Finally, in the last
section, we deduce from the previous sections the existence of a canonical
fully faithful functor from strict \oo-groupoids to Grothendieck
\oo-groupoids.

\smallskip

\begin{tparagr*}{Acknowledgement}
We would like to thank François Métayer, for the multiple helpful
conversations we had about the contractibility of the $n$-disk, and Yves
Lafont, for his suggestion of using the path object of cylinders to show the
contractibility of the $n$-disk.
\end{tparagr*}

\smallskip

\begin{tparagr*}{Notation}
If $C$ is a category, we will denote by $\pref{C}$ the category of
presheaves on~$C$. If 
\[
\xymatrix@C=1pc@R=1pc{
X_1 \ar[dr]_{f_1} & & X_2 \ar[dl]^{g_1} \ar[dr]_{f_2} & &  \cdots & & X_n
\ar[dl]^{g_{n-1}} \\
& Y_1 & & Y_2 & \cdots & Y_{n-1}
}
\]
is a diagram in $C$, we will denote by 
\[ (X_1, f_1) \times_{Y_1} (g_1, X_2, f_2) \times_{Y_2} \cdots
\times_{Y_{n-1}} (g_{n-1}, X_n) \]
its limit. Dually, we will denote by
\[ (X_1, f_1) \amalg_{Y_1} (g_1, X_2, f_2) \amalg_{Y_2} \cdots
\amalg_{Y_{n-1}} (g_{n-1}, X_n) \]
the colimit of the corresponding diagram in the opposite category.
\end{tparagr*}

\section{Strict $\infty$-groupoids and their weak equivalences}

\begin{tparagr}{The globe category}
We will denote by $\G$ the \ndef{globe category}, that is, the category
generated by the graph
\[
\xymatrix{
\Dn{0} \ar@<.6ex>[r]^-{\Ths{1}} \ar@<-.6ex>[r]_-{\Tht{1}} &
\Dn{1} \ar@<.6ex>[r]^-{\Ths{2}} \ar@<-.6ex>[r]_-{\Tht{2}} &
\cdots \ar@<.6ex>[r]^-{\Ths{i-1}} \ar@<-.6ex>[r]_-{\Tht{i-1}} &
\Dn{i-1} \ar@<.6ex>[r]^-{\Ths{i}} \ar@<-.6ex>[r]_-{\Tht{i}} &
\Dn{i} \ar@<.6ex>[r]^-{\Ths{i+1}} \ar@<-.6ex>[r]_-{\Tht{i+1}} &
\cdots
}
\]
under the coglobular relations
\[\Ths{i+1}\Ths{i} = \Tht{i+1}\Ths{i}\quad\text{and}\quad\Ths{i+1}\Tht{i} =
\Tht{i+1}\Tht{i}, \qquad i \ge 1.\]
For $i \ge j \ge 0$, we will denote by $\Ths[j]{i}$ and $\Tht[j]{i}$ the
morphisms from $\Dn{j}$ to $\Dn{i}$ defined by
\[\Ths[j]{i} = \Ths{i}\cdots\Ths{j+2}\Ths{j+1}\quad\text{and}\quad
  \Tht[j]{i} = \Tht{i}\cdots\Tht{j+2}\Tht{j+1}.\]
\end{tparagr}

\begin{tparagr}{Globular sets}
The category of \ndef{globular sets} or \ndef{\oo-graphs} is the category
$\pref{\G}$ of presheaves on $\G$. A globular set $X$ thus consists of a
diagram of sets
\[
\xymatrix{
\cdots \ar@<.6ex>[r]^-{\Gls{i+1}} \ar@<-.6ex>[r]_-{\Glt{i+1}} &
X_{i} \ar@<.6ex>[r]^-{\Gls{i}} \ar@<-.6ex>[r]_-{\Glt{i}} &
X_{i-1} \ar@<.6ex>[r]^-{\Gls{i-1}} \ar@<-.6ex>[r]_-{\Glt{i-1}} &
\cdots \ar@<.6ex>[r]^-{\Gls{2}} \ar@<-.6ex>[r]_-{\Glt{2}} &
X_1 \ar@<.6ex>[r]^-{\Gls{1}} \ar@<-.6ex>[r]_-{\Glt{1}} &
X_0
}
\]
satisfying the globular relations
\[\Gls{i}\Gls{i+1} = \Gls{i}\Glt{i+1}\quad\text{and}\quad\Glt{i}\Gls{i+1} =
\Glt{i}\Glt{i+1}, \qquad i \ge 1.\]
For $i \ge j \ge 0$, we will denote by $\Gls[j]{i}$ and $\Glt[j]{i}$ the
maps from $X_i$ to $X_j$ defined by
\[\Gls[j]{i} = \Gls{j+1}\cdots\Gls{i-1}\Gls{i}\quad\text{and}\quad
  \Glt[j]{i} = \Glt{j+1}\cdots\Glt{i-1}\Glt{i}.\]

If $X$ is a globular set, we will call $X_0$ the set of \ndef{objects} of
$X$ and $X_i$, for $i \ge 0$, the set of \ndef{$i$-arrows}. If $u$ is an
$i$-arrow of $X$ for $i \ge 1$, $\Gls{i}(u)$ (resp.~$\Glt{i}(u)$) will be
called the \ndef{source} (resp.~the \ndef{target}) of $u$. We will often
denote an arrow $u$ of $X$ whose source is $x$ and whose target is $y$ by $u
: x \to y$. We will say that two $n$-arrows $u$ and $v$ are \ndef{parallel}
if either $n = 0$, or $n \ge 1$ and $u, v$ have same source
and same target.
\end{tparagr}

\begin{tparagr}{Strict $\infty$-categories}
An \ndef{\oo-precategory} is a globular set $C$ endowed with maps
\[
  \begin{split}
  \comp_j^i & : \big(C_i, \Gls[j]{i}\big) \times_{C_j} (\Glt[j]{i}, C_i) \to 
      C_i,\quad i > j \ge 0,
      \\
  \Glk{i} & : C_i \to C_{i+1}, \quad i \ge 0,
  \end{split}
\]
such that
\begin{itemize}
    \item 
      for every $i > j \ge 0$ and every $(u, v)$ in $(C_i, \Gls[j]{i})
      \times_{C_j} (\Glt[j]{i}, C_i)$, we have
  \[
  \Gls{i}\big(u \comp_j^i v\big) = 
  \begin{cases}
    \Gls{i}(v), & j = i - 1, \\
    \Gls{i}(u) \comp_j^{i-1} \Gls{i}(v), & j < i - 1,
  \end{cases}
  \]
  and
\[
  \Glt{i}(u \comp_j^i v) = 
  \begin{cases}
    \Glt{i}(u), & j = i - 1, \\
    \Glt{i}(u) \comp_j^{i-1} \Glt{i}(v), & j < i - 1;
  \end{cases}
  \]
  \item for every $i \ge 0$ and every $u$ in $C_i$, we have
  \[
    \Gls{i+1}\Glk{i}(u) = u = \Glt{i+1}\Glk{i}(u).
  \] 
\end{itemize}
If $C$ is an \oo-precategory, for $i \ge j \ge 0$, we will denote by
$\Glk[i]{j}$ the map from $C_j$ to $C_i$ defined by
  \[ \Glk[i]{j} = \Glk{i-1}\cdots\Glk{j+1}\Glk{j}. \]

A \ndef{morphism of \oo-precategories} is a morphism of globular sets
between \oo-precategories which is compatible with the $\ast^i_j$'s and the
$\Glk{i}$'s in an obvious way. We will denote by~$\wpcat$ the category of
\oo-precategories.

A \ndef{strict \oo-category} is an \oo-precategory $C$ satisfying the
following axioms:
\begin{itemize}
  \item Associativity\\
    for every $i > j \ge 0$ and every $(u, v, w)$ in
    $(C_i, \Gls[j]{i}) \times_{C_j} (\Glt[j]{i}, C_i,
    \Gls[j]{i}) \times_{C_j} (\Glt[j]{i}, C_i)$,
    we have
    \[ (u \comp^i_j v) \comp^i_j w = u \comp^i_j (v \comp^i_j w)\text{;} \]
  \item Exchange law\\
    for every $i > j > k \ge 0$ and every $(u, u', v, v')$ in
    \[ (C_i, \Gls[j]{i}) \times_{C_j} (\Glt[j]{i}, C_i, \Gls[k]{i})
    \times_{C_k} (\Glt[k]{i}, C_i, \Gls[j]{i}) \times_{C_j} (\Glt[j]{i},
    C_i),\]
    we have
    \[ (u \comp^i_j u') \comp^i_k (v \comp^i_j v') = (u \comp^i_k v)
    \comp^i_j ( u' \comp^i_k v')\text{;} \]
  \item Identities\\
    for every $i > j \ge 0$ and every $u$ in $C_i$, we have 
       \[
         \Glk[i]{j}\Glt[j]{i}(u) \comp^i_j u = u =
          u \comp^i_j \Glk[i]{j}\Gls[j]{i}(u)\text{;}
        \]
  \item Functoriality of identities\\
    for every $i > j \ge 0$ and every 
    $(u, v)$ in $(C_i, \Gls[j]{i}) \times_{C_j} (\Glt[j]{i} ,C_i)$,
    we have
    \[ \Glk{i}(u \comp^i_j v) = \Glk{i}(u) \comp^{i+1}_j \Glk{i}(v). \]
\end{itemize}

The \ndef{category of strict \oo-categories} is the full subcategory of the
category of \oo-precat\-e\-go\-ries whose objects are strict \oo-categories. 

If $C$ is a strict \oo-category and if $u$ and $v$ are two parallel $n$-arrows of
$C$, the set of \hbox{$(n+1)$}\nbd-arrows from~$u$ to $v$ in $C$ will be denoted by
$\Hom_C(u, v)$.
\end{tparagr}

\begin{paragr}
Let $C$ be a strict \oo-category. To simplify the formulas involving the
operations of~$C$, we will adopt the two following conventions:
\begin{itemize}
  \item If $u$ is an $i$-arrow of $C$ and $v$ is a $j$-arrow of $C$ such
    that $\Gls[k]{i}(u) = \Glt[k]{j}(v)$ for some~$k$ less than $i$ and $j$,
    then we will denote by $u \comp_k v$ the $m$-arrow $\Glk[m]{i}(u)
    \comp^m_k \Glk[m]{j}(v)$, where $m$ is the greatest integer between $i$
    and $j$.
  \item If $u$ is a $j$-arrow of $C$, we will denote by $\Glid{u}$ the
    $(j+1)$-arrow $\Glk{j}(u)$ or, more generally, the $i$-arrow $\Glk[i]{j}(u)$
    for any $i \ge j$ if the context makes the value of $i$ clear.
\end{itemize}
\end{paragr}

\begin{tparagr}{Strict \oo-groupoids}
Let $C$ be a strict \oo-category and let $u : x \to y$ be an $i$-arrow
of~$C$ for $i \ge 1$. A \ndef{$\comp^i_{i-1}$-inverse}, or briefly an
\ndef{inverse}, of $u$ is an $i$-arrow $u^{-1} : y \to x$ of $C$ such that 
\[
u^{-1} \comp_{i-1} u = \Glid{x}
\quad\text{and}\quad
u \comp_{i-1} u^{-1} = \Glid{y}.
\]
If such an inverse exists, it is unique and the notation $u^{-1}$ is thus
unambiguous.

A \ndef{strict \oo-groupoid} is a strict \oo-category whose $i$-arrows, $i
\ge 1$, are invertible. The existence of $\comp^i_{i-1}$-inverses implies
the existence of $\comp^i_j$-inverses (in an obvious sense) for every $i >
j \ge 0$ (see Proposition 2.3 of \cite{AraMetBrown} for details).

The \ndef{category of strict \oo-groupoids} is the full subcategory of the
category of strict \oo-cat\-e\-go\-ries whose objects are strict
\oo-groupoids. We will denote it by $\wgpds$. Note that a morphism of strict
\oo-groupoids automatically preserves the inverses.
\end{tparagr}

\begin{tparagr}{Homotopy groups of strict \oo-groupoids}
Let $G$ be a strict \oo-groupoid. An $n$-arrow $u$ of $G$ is \ndef{homotopic}
to another $n$-arrow $v$ of $G$ if there exists an $(n+1)$-arrow from~$u$ to
$v$ in $G$. This obviously implies that the arrows $u$ and $v$ are parallel.
If $u$ is homotopic to~$v$, we will write $u \sim v$. The relation~$\sim$ is
an equivalence relation on $G_n$. It is moreover compatible with the
composition $\comp^n_{n-1} : G_n \times_{G_{n-1}} G_n \to G_n$.

The set of \ndef{connected components} of $G$ is
\[ \pi_0(G) = G_0/{\sim}. \]
If $n \ge 1$ and $u, v$ are two parallel $n$-arrows of $G$, we will denote
\[
\pi_n(G, u, v) = \Hom_G(u, v)/{\sim}
\quad\text{and}\quad
\pi_n(G, u) = \pi_n(G, u, u).\
\]
Note that the composition $\comp^n_{n-1}$ induces a group structure on
$\pi_n(G, u)$. For $n \ge 1$ and $x$~an object of $G$, the \ndef{$n$-th
homotopy group of $G$ at $x$} is
\[ \pi_n(G, x) = \pi_n(G, \Glk[n-1]{0}(x)). \]
It is immediate that $\pi_0$ induces a functor from the category of strict
\oo-groupoids to the category of sets, and that $\pi_n$, for $n \ge 1$, induces
a functor from the category of pointed strict \oo-groupoids to the category
of groups. Moreover, by the Eckmann-Hilton argument, the groups $\pi_n(G, x)$
are abelian when $n \ge 2$.
\end{tparagr}

\begin{tparagr}{Weak equivalences of strict \oo-groupoids}
A morphism $f : G \to H$ of strict \oo-groupoids is a \ndef{weak
equivalence} if
\begin{itemize}
  \item the map $\pi_0(f) : \pi_0(G) \to \pi_0(H)$ is a bijection;
  \item for all $n \ge 1$ and all object $x$ of $G$, the morphism 
    \[ \pi_n(f, x) : \pi_n(G, x) \to \pi_n(H, f(x)) \]
    is a group isomorphism.
\end{itemize}
\end{tparagr}

\begin{prop}\label{prop:def_w_eq}
Let $f : G \to H$ be a morphism of strict \oo-groupoids. The following conditions are
equivalent:
\begin{enumerate}
\item\label{item:w_eq} $f$ is a weak equivalence of strict \oo-groupoids;
\item $\pi_0(f) : \pi_0(G) \to \pi_0(H)$ is a bijection and for
all $n \ge 1$ and every $(n-1)$-arrow~$u$ of~$G$, $f$
 induces a bijection
 \[ \pi_n(G, u) \to \pi_n(H, f(u))\text{;} \]
\item $\pi_0(f) : \pi_0(G) \to \pi_0(H)$ is a bijection and for
all $n \ge 1$ and every pair $(u, v)$ of parallel $(n-1)$-arrows of~$G$, $f$
 induces a bijection
 \[ \pi_n(G, u, v) \to \pi_n(H, f(u), f(v))\text{;} \]
\item\label{item:folk_w_eq} 
$\pi_0(f) : \pi_0(G) \to \pi_0(H)$ is surjective and for all $n \ge 1$ and
every pair $(u, v)$ of parallel $(n-1)$-arrows of $G$, $f$ induces a surjection
\[ \pi_n(G, u, v) \to \pi_n(H, f(u), f(v))\text{.} \]
\end{enumerate}
\end{prop}

\begin{proof}
See Proposition 1.7 of \cite{AraMetBrown}.
\end{proof}

\begin{tparagr}{Weakly contractible strict \oo-groupoids}
A strict \oo-groupoid $G$ is said to be \ndef{weakly contractible} if the
unique morphism from $G$ to the terminal strict \oo-groupoid is a weak
equivalence. In other words, $G$ is weakly contractible if
\begin{itemize}
  \item the set $\pi_0(G)$ is trivial;
  \item for all $n \ge 1$ and all object $x$ of $G$, the group $\pi_n(G, x)$
    is trivial.
\end{itemize}
\end{tparagr}

\begin{prop}\label{prop:def_contr}
A strict \oo-groupoid $G$ is weakly contractible if and only if $G$ is
non-empty and for every $n \ge 1$ and every pair $(u, v)$ of parallel
$(n-1)$-arrows of $G$, there exists an $n$-arrow from $u$ to $v$ in $G$.
\end{prop}

\begin{proof}
This is exactly the content of the equivalence $(\ref{item:w_eq})
\Leftrightarrow (\ref{item:folk_w_eq})$ of Proposition \ref{prop:def_w_eq}
applied to the unique morphism from $G$ to the terminal strict \oo-groupoid.
\end{proof}

\section{Grothendieck $\infty$-groupoids}

In this section, we recall briefly the definition of Grothendieck
\oo-groupoids and we state our main result. We encourage the reader to read
Sections 1 and 2 of \cite{AraHomGr} for more explanation and examples.

\begin{tparagr}{Globular sums}
A \ndef{table of dimensions} is a table
\[
\tabdim,
\]
where $k \ge 1$, consisting of nonnegative integers satisfying 
\[ i_l > i'_l <  i_{l+1}, \qquad 1 \le l \le n - 1. \]
The integer $k$ is called the \ndef{width} of $T$.
The \ndef{dimension} of such a table is the greatest integer appearing in
the table.

Let $(C, F)$ be a category under $\G$, i.e., a category $C$ endowed with
a functor $F : \G \to C$. We will often denote in the same way the objects
and morphisms
of $\G$ and their image by the functor $F$. Let
\[ T = \tabdim \]
be a table of dimensions. The \ndef{globular sum} in $C$ associated to $T$
(if it exists) is the iterated amalgamated sum
\[ (\Dn{i_1}, \Ths[i'_1]{i_1}) \amalgd{i'_1} (\Tht[i'_1]{i_2}, \Dn{i_2},
\Ths[i'_2]{i_2}) \amalgd{i'_2} \cdots
\amalgd{i'_{k-1}} (\Tht[i'_{k-1}]{i_k}, \Dn{i_k}) \]
in $C$, i.e., the colimit of the diagram
\[
\xymatrix@R=.2pc@C=1pc{
\Dn{i_1} &  & \Dn{i_2} &  & \Dn{i_3} &        & \Dn{i_{k - 1}} & & \Dn{i_{k}} \\
  &  &   &  &   & \cdots &     & & \\
  & \Dn{i'_1}
  \ar[uul]^{\Ths[i'_1]{i_1}}
  \ar[uur]_{\negthickspace \Tht[i'_1]{i_2}}  &   &
  \Dn{i'_2}' 
  \ar[uul]^{\Ths[i'_2]{i_2}\negthickspace}
  \ar[uur]_{\Tht[i'_2]{i_3}}  &  &  & &
\Dn{i'_{k-1}} \ar[uul]^{\Ths[i'_{k-1}]{i_{k-1}}}  \ar[uur]_{\Tht[i'_{k-1}]{i_k}}
& &
}
\]
in $C$.
We will denote it briefly by
\[
\Dn{i_1} \amalgd{i'_1} \Dn{i_2} \amalgd{i'_2} \cdots
\amalgd{i'_{k-1}} \Dn{i_k}.
\]
\end{tparagr}

\begin{tparagr}{Globular extensions}
A category $C$ under $\G$ is said to be a \ndef{globular extension} if all
the globular sums exist in $C$, i.e., if for every table of dimensions
$T$, the globular sum associated to $T$ exists in $C$. 

If $C$ and $D$ are two globular extensions, a \ndef{morphism of globular
extensions} from $C$ to~$D$ is a functor from $C$ to~$D$ under $\G$
(that is, such that the triangle
\[
\xymatrix@C=1pc@R=1pc{
& \G \ar[ld] \ar[rd] \\
C \ar[rr] & & D
}
\]
commutes) which preserves globular sums. Such a functor will also be called a
\ndef{globular functor}. 
\end{tparagr}

\begin{tparagr}{The globular extension $\Thz$}\label{paragr:thz}
We will consider the category $\pref{\G}$ as a category under~$\G$ by using
the Yoneda functor. If $T$ is a table of dimensions, we will denote by~$G_T$
the globular sum associated to $T$ in $\pref{\G}$.

The category $\Thz$ is the category defined in the following way:
\begin{itemize}
  \item the objects of~$\Thz$ are the table of dimensions;
  \item if $S$ and $T$ are two objects of $\Thz$, then
    \[ \Hom_{\Thz}(S, T) = \Hom_{\pref{\G}}(G_S, G_T). \]
\end{itemize}
By definition of $\Thz$, there is a canonical fully faithful functor $\Thz
\to \pref{\G}$. This functor is moreover injective on objects and $\Thz$ can
be considered as a full subcategory of $\pref{\G}$.

The functor $\G \to \pref{\G}$ factors through $\Thz$ and we get a functor
$\G \to \Thz$. The category~$\Thz$ will always be considered as a category
under $\G$ using this functor. By definition, $\Thz$ is a globular
extension.

The globular extension $\Thz$ is the initial globular extension in the
following sense: for every globular extension $C$, there exists a globular
functor $\Thz \to C$, unique up to a unique natural transformation (see
Proposition 3.2 and paragraph 3.3 of \cite{AraThtld}). More precisely, if
$C$ is a globular extension, the choice of a globular functor $\Thz \to C$
amounts to the choice, for every table of dimensions $T$, of a globular sum
associated to $T$ in $C$ (this globular sum being only defined up to a
canonical isomorphism).
\end{tparagr}

\vbox{
\begin{rem}\ 
\begin{myenumeratep}
  \item The $G_T$'s are exactly the globular sets associated to finite
    planar rooted trees by Batanin in \cite{BataninWCat}. These globular
    sets are sometimes called \emph{globular pasting schemes}.
  \item The category $\Thz$ was first introduced by Berger in
    \cite{BergerNerve} in terms of planar rooted trees.
\end{myenumeratep}
\end{rem}
}

\begin{tparagr}{Globular theories}
A globular extension $C$ is called a \ndef{globular theory} if any globular
functor $\Thz \to C$ is bijective on objects. If $C$ is a globular theory,
then there exists a unique globular functor $\Thz \to C$.

A \ndef{morphism of globular theories} is a morphism of globular extensions
between globular theories. Note that if $C$ and $D$ are globular theories, a
morphism from $C$ to $D$ is nothing but a functor from~$C$ to~$D$
under~$\Thz$.
\end{tparagr}

\begin{tparagr}{Globular presheaves}\label{paragr:glob_pref}
Let $C$ be a globular theory. A \ndef{globular presheaf} on $C$, or
\ndef{model} of $C$, is a presheaf~$X$ on~$C$ which sends globular sums to
globular products (globular products being the notion dual to globular
sums). In other words, a presheaf~$X$ is a globular presheaf if, for every
table of dimensions
\[ \tabdim, \]
the canonical map
\[
X(\Dn{i_1} \amalgd{i'_1} \cdots \amalgd{i'_{k-1}} \Dn{i_k}) \to
X_{i_1} \times_{X_{i'_1}} \cdots \times_{X_{i'_{k-1}}} X_{i_k}
\]
is a bijection. We will denote by $\Mod{C}$ the full subcategory of the
category of presheaves on $C$ whose objects are the globular presheaves on
$C$.

If $C \to D$ is a morphism of globular theories, then the inverse image
functor from presheaves on $D$ to presheaves on $C$ restricts to a functor
from globular presheaves on~$D$ to globular presheaves on $C$.
\end{tparagr}

\begin{tparagr}{Globularly parallel arrows and liftings}
Let $C$ be a globular extension. Two morphisms $f, g : \Dn{n} \to X$ of $C$
are said to be globularly parallel if either $n = 0$, or $n \ge 1$ and
\[ f\Ths{n} = g\Ths{n} \quad\text{and}\quad f\Tht{n} = g\Tht{n}. \]

Let now $(f, g) : \Dn{n} \to X$ be a pair of morphisms of $C$.
A \ndef{lifting} of the pair $(f, g)$ is a morphism $h : \Dn{n+1} \to X$
such that
\[ h\Ths{n} = f \quad\text{and}\quad h\Tht{n} = g. \]
The existence of such a lifting obviously implies that $f$ and $g$ are
globularly parallel.
\end{tparagr}

\begin{tparagr}{Admissible pairs}\label{def:adm}
Let $C$ be a globular extension.
A pair of morphisms 
\[ (f, g) : \Dn{n} \to S \]
is said to be \ndef{$(\infty, 0)$-admissible}, or briefly \ndef{admissible},
if
\begin{itemize}
\item the morphisms $f$ and $g$ are globularly parallel;
\item the object $S$ is a globular sum;
\item the dimension of $S$ (as a globular sum) is less than or equal to $n + 1$.
\end{itemize}
\end{tparagr}

\begin{tparagr}{Contractible globular extensions}
A globular extension $C$ is \ndef{$(\infty, 0)$\nbd-con\-tracti\-ble},
or briefly \ndef{contractible}, if every admissible pair of $C$ admits a
lifting. If moreover, these liftings are unique, we will say that $C$ is
\ndef{canonically $(\infty, 0)$-contractible}, or briefly \ndef{canonically
contractible}.
\end{tparagr}

\begin{rem}\label{rem:def_adm}
The last condition in the definition of an admissible pair was not part of
Grothendieck's original definition and was introduced by us in
\cite{AraHomGr}. This condition is needed to make canonical the functor from
strict \oo-groupoids to Grothendieck \oo-groupoids obtained in
Theorem~\ref{thm:strweak} or, equivalently, to make the globular
extension~$\Thtld$ (defined in paragraph \ref{paragr:def_thtld}) a
\emph{canonically} contractible globular extension (see
Theorem~\ref{thm:thtld_can_contr} and Remark~\ref{rem:thtld_can_contr}).
\end{rem}

\begin{tparagr}{Free globular extensions}
A \ndef{cellular tower} of globular extensions is a tower of globular
extensions
\[ C_0 = \Thz \to C_1 \to \cdots \to C_n \to \cdots\pbox{,} \]
endowed, for each $n \ge 0$, with a set $A_n$ of admissible pairs of $C_n$,
such that $C_{n+1}$ is the globular extension obtained from $C_n$ by
formally adding a lifting to each admissible pair in $A_n$.

We will say that a globular extension $C$ is \ndef{free} if $C$ is the
colimit of some cellular tower of globular extensions. If $C$ is a free
globular extension, then $C$ is a globular theory.
\end{tparagr}

\begin{prop}\label{prop:coh_weak_init}\ 
Let $C$ be a free globular extension.
\begin{enumerate}
  \item For any contractible globular extension $D$, there exists a globular
    functor from $C$ to $D$.
  \item For any canonically contractible globular theory $D$, there exists a
    unique globular functor $C$ to $D$.
\end{enumerate}
\end{prop}

\begin{proof}
See Proposition 2.14 of \cite{AraHomGr} for the first point. The second
point is related to Remark 2.14.1 of loc.~cit. Let us prove it (this will
also give a proof of the first point).

Let $(C_n, A_n)$ be any cellular tower such that $C$ is the colimit of the
$C_n$'s. For $n \ge 0$, let us denote by $i_{n+1}$ the globular functor from
$C_n$ to $C_{n+1}$. By the universal property of the colimit, a globular
functor $C \to D$ is given by a system of globular functors $F_n : C_n \to  D$
such that $F_{n+1}i_n = F_n$. (Note that all the functors involved are
functors under $\Thz$ and are hence automatically globular.) By the
universal property defining $C_{n+1}$ from~$C_n$, such a system
is uniquely determined by $F_0 : C_0 = \Thz \to D$ and by the
choice, for every pair $(f, g)$ of $A_n$, of a lifting of the admissible
pair $(F_n(f), F_n(g))$ in $D$.

It follows immediately that if $D$ is contractible, such a system exists and
that if $F_0$ is fixed (which is the case if $D$ is a globular theory) and
that $D$ is canonically contractible, such a system is unique.
\end{proof}

\begin{tparagr}{Coherators}
An \ndef{$(\infty, 0)$-coherator}, or briefly a \ndef{coherator}, is a
globular extension which is free and $(\infty, 0)$-contractible.
\end{tparagr}

\begin{tparagr}{$\infty$-groupoids of type~$C$}
Let $C$ be a coherator. An \ndef{\oo-groupoid of type~$C$} is a globular
presheaf on $C$. The category of \oo-groupoids of type $C$ is the category
$\Mod{C}$. This category will be denoted more suggestively by~$\wgpdC{C}$.
\end{tparagr}

\goodbreak
We can now state our main result:
\begin{thm*}
Let $C$ be a coherator. There exists a canonical functor
\[ \wgpds \to \wgpdC{C}. \]
Moreover, this functor is fully faithful.
\end{thm*}
\noindent This result will be proved in the very last section of the
article.

\section{The globular extension $\Thtld$}

\begin{tparagr}{Free strict $\infty$-groupoids}
Let $U : \wgpds \to \pref{\G}$ be the forgetful functor sending a strict
\oo-groupoid to its underlying globular set. This functor $U$ admits a left
adjoint $L : \pref{\G} \to \wgpds$ which by definition sends a globular set
to the \ndef{free strict \oo-groupoid} on this globular set.
\end{tparagr}

\begin{tparagr}{The globular extension $\Thtld$}\label{paragr:def_thtld}
Recall that if $T$ is a table of dimensions, we defined in paragraph
\ref{paragr:thz} an associated globular set $G_T$.

The category $\Thtld$ is the category defined in the following way:
\begin{itemize}
  \item the objects of $\Thtld$ are the table of dimensions;
  \item if $S$ and $T$ are two objects of $\Thtld$, then
    \[ \Hom_{\Thtld}(S, T) = \Hom_{\wgpds}(L(G_S), L(G_T)). \]
\end{itemize}
By definition of $\Thtld$, there are canonical functors
\[ \Thz \to \Thtld \to \wgpds. \]
The functor $\Thtld \to \wgpds$ is by definition fully faithful. It is
easily seen to be injective on objects and $\Thtld$ can thus be considered
as a full subcategory of the category of strict \oo-groupoids.

The category $\Thtld$ will always be considered as a category under $\G$ by
using the functor $\G \to \Thz \to \Thtld$. It follows immediately from the
fact that $L$ is a left adjoint and hence preserves colimits that $\Thtld$
is a globular extension and hence a globular theory.
\end{tparagr}

\begin{prop}\label{prop:glob_pref_thtld}
The category of globular presheaves on $\Thtld$ is canonically equivalent to
the category of strict \oo-groupoids. 
\end{prop}

\begin{proof}
See Propositions 3.21 and 3.22 of \cite{AraThtld}.
\end{proof}

\begin{tparagr}{Towards our canonical fully faithful functor}
The next three sections are dedicated to proving that the globular extension
$\Thtld$ is canonically contractible (see the introduction for details on
the different steps). Assuming this fact, we can easily get half of our main
result. Indeed, by applying Proposition~\ref{prop:coh_weak_init} to
$\Thtld$, we get that if $C$ is a coherator, there exists a unique globular
functor $\Thtld \to C$. This globular functor induces a functor 
\[ \wgpds \cong \Mod{\Thtld} \to \Mod{C} = \wgpdC{C}, \] 
and we thus obtain a canonical functor from strict \oo-groupoids to
\oo-groupoids of type~$C$.

The second half of the result, namely the fact that this functor is fully
faithful, will follow from the developments of Section \ref{sec:ff}.
\end{tparagr}

\section{Homotopy theory of strict \oo-groupoids}

\begin{tparagr}{Trivial fibrations of strict \oo-groupoids}
Recall that a morphism of presheaves is said to be a \ndef{trivial
fibration} if it has the right lifting property with respect to every
monomorphism. This defines in particular a notion of \ndef{trivial
fibration} of globular sets.

A morphism of strict \oo-groupoids is said to be a \ndef{trivial fibration}
if its underlying morphism of globular sets is a trivial fibration.
\end{tparagr}

\begin{tparagr}{Free strict $\infty$-groupoids}
Recall that we denote by $U : \wgpds \to \pref{\G}$ the forgetful functor
from strict \oo-groupoids to globular sets and by $L : \pref{\G} \to \wgpds$
its left adjoint.
\end{tparagr}

\begin{tparagr}{Disks}
We will consider the Yoneda functor $\G \to \pref{\G}$ as an inclusion. In
particular, for every $n \ge 0$, we have a globular set $\Dn{n}$. If $X$ is a
globular set, a morphism $\Dn{n} \to X$ corresponds, by the Yoneda lemma, to an
$n$-arrow of $X$.

For $n \ge 0$, the strict \oo-groupoid $L(\Dn{n})$ will be denoted by
$\Dnt{n}$. It follows from the fact that $L$ is a left adjoint to $U$ that
if $G$ is a strict \oo-groupoid, a morphism $\Dnt{n} \to G$ corresponds to
an $n$-arrow of $G$.

For $i \ge j \ge 0$, we have two morphisms $\Thst[j]{i}, \Thtt[j]{i} :
\Dnt{j} \to \Dnt{i}$ defined by
\[
\Thst[j]{i} = L(\Ths[j]{i})
\quad\text{and}\quad
\Thtt[j]{i} = L(\Tht[j]{i}).
\]
When $j = i - 1$, we will denote
$\Thst{i} = \Thst[i-1]{i}$
and
$\Thtt{i} = \Thtt[i-1]{i}$.
\end{tparagr}

\begin{tparagr}{Spheres}
We define by induction on $n \ge 0$ a globular set $\Sn{n-1}$ endowed with a
morphism $i_n : \Sn{n-1} \to \Dn{n}$ in the following way. For $n = 0$, we
set
 \[ \Sn{-1} = \varnothing, \]
the empty globular set, and we define
 \[ i_0 : \Sn{-1} \to \Dn{0} \] 
as the unique morphism from the initial object to $\Dn{0}$. For $n \ge 1$,
we set
\[
\Sn{n-1} = (\Dn{n-1}, i_{n-1}) \amalg_{\Sn{n-2}} (i_{n-1}, \Dn{n-1})
\]
and
\[ i_n = (\Tht{n}, \Ths{n}) : \Sn{n-1} \to \Dn{n}. \]

The globular set $\Sn{n-1}$ can be described concretely as the sub-globular
set of $\Dn{n}$ obtained from $\Dn{n}$ by removing the unique $n$-arrow. In
particular, if $X$ is a globular set and~$n \ge 1$, a morphism $\Sn{n-1} \to
X$ corresponds to a pair of parallel $(n-1)$-arrows of $X$.

If $n \ge 0$, the strict \oo-groupoid $L(\Sn{n-1})$ will be denoted by
$\Snt{n-1}$. Since $L$ is a left adjoint and hence preserves pushouts, the
$\Snt{n-1}$'s can be constructed from the $\Dnt{m}$'s using a similar
induction. In particular, we have
\[ \Snt{n-1} = \Dnt{n-1} \amalg_{\Snt{n-2}} \Dnt{n-1}. \]
It follows from the fact that $L$ is a left adjoint to $U$ that if $G$ is a
strict \oo-groupoid and~$n \ge 1$, a morphism $\Snt{n-1} \to G$ corresponds
to a pair of parallel $(n-1)$-arrows of $G$.
\end{tparagr}

\vbox{
\begin{prop}\label{prop:triv_fib}\ 
\begin{enumerate}
  \item A morphism $X \to Y$ of globular sets is a trivial fibration if
    and only if it has the right lifting property with respect to the
    morphisms $\Sn{n-1} \to \Dn{n}$, $n \ge 0$.
  \item A morphism $G \to H$ of strict \oo-groupoids is a trivial fibration
    if and only if it has the right lifting property with respect to the
    morphisms $\Snt{n-1} \to \Dnt{n}$, $n \ge 0$.
\end{enumerate}
\end{prop}
}

\begin{proof}
The category $\G$ is a direct category. Moreover, the concrete description
of $\Sn{n-1}$ shows that the inclusion $\Sn{n-1} \to \Dn{n}$ is nothing but
the inclusion $\eDn{n} \to \Dn{n}$, where $\eDn{n}$ is the boundary of
$\Dn{n}$ defined in terms of the direct structure of $\G$ (see for instance
paragraph~8.1.30 of~\cite{Cisinski} for a definition). The first assertion
then follows from Proposition~8.1.37 of~\cite{Cisinski}.

The second assertion follows formally from the first one by using the fact
that $L$ is a left adjoint to $U$.
\end{proof}

\begin{rem}
A similar result holds for strict \oo-categories. In particular, the trivial
fibrations of \cite{LMW} can be described without reference to spheres and
disks.
\end{rem}

\begin{thm}[Brown-Golasi\'nski, Ara-Métayer]\label{thm:BG}
The weak equivalences and the trivial fibrations of strict \oo-groupoids
define a combinatorial model category structure on the category of strict
\oo-groupoids. Moreover, every strict \oo-groupoid is fibrant in this model
category structure.
\end{thm}

\begin{proof}
Theorem 3.19 of \cite{AraMetBrown} asserts the existence of a combinatorial
model category structure on $\wgpds$. The weak equivalences of this model
category structure coincide with the weak equivalences of strict
\oo-groupoids defined in this article by Proposition~4.1 of loc.~cit. The
trivial fibrations of this model category structure are the morphisms having
the right lifting property with respect to the $\Snt{n-1} \to \Dnt{n}$. They
coincide with the trivial fibrations of strict \oo-groupoids defined in this
article by Proposition~\ref{prop:triv_fib}. The fact that every strict
\oo-groupoid is fibrant follows from Theorem~5.1 of~\cite{LMW}.
\end{proof}

\begin{rem}
This model category was first defined by Brown and Golasi\'nski in
\cite{BrownGolas} in terms of crossed complexes (crossed complexes are
equivalent to strict \oo-groupoids by the main result of
\cite{BrownHigginsGpdCC}). We will hence call it the Brown-Golasi\'nski model
category structure. An alternative proof and a direct description are given
in \cite{AraMetBrown} in terms of the model category structure on strict
\oo-categories defined in \cite{LMW}.
\end{rem}

\begin{rem}\label{rem:Dnt_cof}
By Proposition \ref{prop:triv_fib}, for every $n \ge 0$, the morphism
$\Snt{n-1} \to \Dnt{n}$ is a cofibration (in the Brown-Golasi\'nski model
category structure). It follows that for every~$n \ge 1$, the morphisms
$\Thst{n}, \Thtt{n} : \Dnt{n-1} \to \Dnt{n}$ are also cofibrations.  Indeed,
these morphisms are obtained as compositions
\[
\Dnt{n-1} \to \Snt{n-1} = \Dnt{n-1} \amalg_{\Snt{n-2}} \Dnt{n-1} \to \Dnt{n},
\]
where the first arrow is one of the two canonical morphisms. But these
canonical morphisms are both pushouts of $\Snt{n-2} \to \Dnt{n-1}$.
\end{rem}

\begin{tparagr}{Path objects}
Let us fix some terminology about path objects. Let $\M$ be a model category
and let $B$ be an object of $\M$. A \ndef{path object} of $B$ in $\M$ is an
object $P$ of~$\M$ endowed with a factorization
\[
\xymatrix{
B \ar[r]^r & P \ar[r]^-{(p_1, p_0)} & B \times B
}
\]
of the diagonal of $B$ as a weak equivalence followed by a fibration. 

Let $P$ be such a path object and let $f, g : A \to B$ be two morphisms of
$\M$.  A \ndef{right homotopy} from $f$ to $g$ using~$P$ is a morphism $h : A
\to P$ of $\M$ such that $p_0h = f$ and~$p_1h = g$.  The existence of such a
homotopy implies that $f$ and $g$ become equal in the homotopy category
of~$\M$. Note that we do not need the morphism $P \to B \times B$
to be a fibration for this property to hold.
\end{tparagr}

The rest of this section is devoted to the description of a functorial path
object for the Brown-Golasi\'nski model category structure.

\begin{tparagr}{Notation for iterated sources and targets}
Let $G$ be a strict \oo-groupoid and let $u$ be an $n$-arrow of $G$. We will
use the following notation for the iterated sources, targets and
identities of the $n$\nbd-arrow~$u$:
\[
u^\flat_i = 
\begin{cases}
\Gls[i]{n}(u), & \text{if $i \le n$},\\
\Glk[i]{n}(u), & \text{if $i > n$},
\end{cases}
\quad\text{and}\quad
u^\sharp_i = 
\begin{cases}
\Glt[i]{n}(u), & \text{if $i \le n$},\\
\Glk[i]{n}(u), & \text{if $i > n$}.
\end{cases}
\]
Note that by definition, we have $u = u^\flat_n = u^\sharp_n$. Here is an
illustration of this notation in the case $n = 3$:
\[
\newcommand\st\scriptstyle
\xymatrix@C=4pc@R=4pc{
\st{u^\flat_0} \ar@/^3ex/[r]_(.47){}="0"^(.53){}="10"^{u^\flat_1}
    \ar@/_3ex/[r]_(.47){}="1"^(.53){}="11"_{u^\sharp_1}
    \ar@<2ex>@2"0";"1"_{}="2"^{u^\sharp_2} \ar@<-2ex>@2"10";"11"^{}="3"_{u^\flat_2}
    \ar@3"3";"2"_{}^u
    & \st{u^\sharp_0} \pbox{.}
}
\]
\end{tparagr}
 
\begin{tparagr}{Cylinders}\label{paragr:cylinders}
Let $G$ be a strict \oo-groupoid. Let $n \ge 0$ and let $u, v$ be two
$n$-arrows of $G$. An \ndef{$n$-cylinder} $z$ from $u$ to $v$ in $G$,
denoted by $z : u \cto v$, consists of:
\begin{itemize}
  \item for every $1 \le i \le n$, two $i$-arrows 
    $z^\flat_i, z^\sharp_i$ of $G$;
  \item an $(n + 1)$-arrow $z_{n+1}$ of $G$ (which will also be
    denoted by $z^\sharp_{n+1}$ and $z^\flat_{n+1}$ for the purpose of
    getting homogeneous formulas),
\end{itemize}
whose sources and targets are given, for $1 \le i \le n + 1$ and $\epsilon =
\flat, \sharp$, by the following formulas:
\[
\begin{split}
\Gls{i}(z^\epsilon_i) & = z^\sharp_{i-1} \comp_{i-2}
\big(z^\sharp_{i-2} \comp_{i-3} \big(\cdots \comp_1
\big(z^\sharp_1 \comp_0 u^\epsilon_{i-1}\big)\big)\big),
\\
\Glt{i}(z^\epsilon_i) & = \big(\big(\big(v^\epsilon_{i-1} \comp_0
z^\flat_1\big)\comp_1 z^\flat_2\big) \comp_2 \cdots\big)
\comp_{i-2} z^\flat_{i-1}.
\end{split}
\]
Note that if $i = n + 1$, these formulas do not depend on the value of
$\epsilon$ and the definition hence makes sense. We will denote by
$\Gamma_n(G)$ the set of $n$-cylinders in $G$. Here are the diagrams
representing $n$-cylinders for $n = 0, 1, 2$:
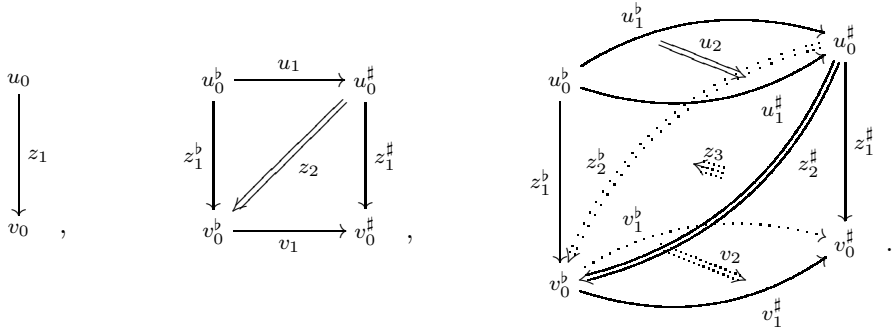
\[
\newcommand\st\scriptstyle
\xygraph{
!{<0cm,0cm>;<1cm,0cm>:<0cm,1cm>::}
!{(0,0) }*+{\st u_0}="u"
!{(0,-2) }*+{\st v_0\pbox{,}}="v"
"u":"v"^{z_1}
}
\qquad
\qquad
\xygraph{ 
!{<0cm,0cm>;<1cm,0cm>:<0cm,1cm>::}
!{(3,0) }*+{\st u^\flat_0}="ufz"
!{(5,0)}*+{\st u^\sharp_0}="usz"
!{(3,-2) }*+{\st v^\flat_0}="vfz"
!{(5,-2)}*+{\st v^\sharp_0\pbox{,}}="vsz"
"ufz":"vfz"_{z^\flat_1}
"usz":"vsz"^{z^\sharp_1}
"ufz":"usz"^{u_1}
"vfz":"vsz"_{v_1}
"usz":@{=>}"vfz"^{z_2}
}
\qquad
\qquad
\xygraph{ 
!{<0cm,0cm>;<1.5cm,0cm>:<0cm,.9cm>::} 
!{(0,0) }*+{\st u^\flat_0}="ufz"
!{(0,-3) }*+{\st v^\flat_0}="vfz"
!{(2.5,0.6)}*+{\st u^\sharp_0}="usz"
!{(2.5,-2.4)}*+{\st v^\sharp_0\pbox{.}}="vsz"
!{(.8,0.6)}*+{}="ufo"
!{(1.7,0)}*+{}="uso"
!{(.8,-2.4)}*+{}="vfo"
!{(1.7,-3)}*+{}="vso"
!{(1.1,-1.2)}*+{}="codz"
!{(1.5,-1.4)}*+{}="domz"
"ufz":@/^.5cm/"usz"^(.3){u^\flat_1}
"ufz":@/_.55cm/"usz"_(.7){u^\sharp_1}
"ufo":@{=>}"uso"^{u_2}
"vfz":@{.>}@/^.5cm/"vsz"^(.3){v^\flat_1}
"vfz":@/_.55cm/"vsz"_(.7){v^\sharp_1}
"vfo":@{:>}"vso"^(.7){v_2}
"ufz":"vfz"_{z^\flat_1}
"usz":"vsz"^{z^\sharp_1}
"usz":@{=>}@/^.7cm/"vfz"^(.3){z^\sharp_2}
"usz":@{:>}@/_.7cm/ "vfz"_(.7){z^\flat_2}
"domz":@3{.>}"codz"_{z_3}
}
\]
\end{tparagr}

\begin{rem}
This notion of $n$-cylinder was originally defined by Métayer (in the more
general context of strict \oo-categories) in \cite{MetPolRes}. In
\cite{LafMetPol}, Lafont and Métayer give a nice inductive reformulation of
the definition (an $n$-cylinder is defined in terms of $(n-1)$\nbd-cylinders
in another strict \oo-category). These $n$-cylinders are also studied in
\cite{LMW} where they play a crucial role.
\end{rem}

\begin{tparagr}{The strict \oo-groupoid of cylinders}
Let $G$ be a strict \oo-groupoid. If $z : u \cto v$ is an $n$-cylinder with
$n > 0$, we define two $(n-1)$-cylinders
\[
\Ths{n}(z) : u^\flat_{n-1} \cto v^\flat_{n-1}
\quad\text{and}\quad
\Tht{n}(z) : u^\sharp_{n-1} \cto v^\sharp_{n-1}
\]
by setting
\[
\Ths{n}(z)^\epsilon_i = \Tht{n}(z)^\epsilon_i = z^\epsilon_i,
\quad 1 \le i \le n - 1, \quad \epsilon = \flat, \sharp,
\]
\[
 \Ths{n}(z) = z^\flat_n
 \quad\text{and}\quad
 \Tht{n}(z) = z^\sharp_n.
\]
This defines the structure of a globular set on the $\Gamma_n(G)$'s. The
resulting globular set will be denoted by $\Gamma(G)$. Métayer showed in an
appendix to \cite{MetPolRes} that $\Gamma(G)$ is naturally endowed with the
structure of a strict \oo-category (see also Appendix A of \cite{LMW}).
Moreover, this strict \oo-category is an \oo-groupoid (see Lemma 3.11 of
\cite{AraMetBrown}). From now on, we will always consider $\Gamma(G)$
endowed with this structure of a strict \oo-groupoid. The only result we
will need that uses the precise definition of this structure is
Theorem~\ref{thm:gamma_cyl}.
\end{tparagr}

\begin{tparagr}{The path object of cylinders}
Let $G$ be a strict \oo-groupoid. If $u$ is an $n$-arrow of $G$, we define
an $n$-cylinder $\tau u : u \cto u$ by setting:
\[ (\tau u)^\epsilon_i = 1_{u^\epsilon_{i-1}}, \quad 1 \le i \le n + 1, \quad
\epsilon = \flat, \sharp. \]
We obtain this way a factorization
\[ G \to \Gamma(G) \to G \times G \]
of the diagonal of $G$ in the category of globular sets. The first
morphism sends an $n$-arrow~$u$ to the $n$-cylinder $\tau u : u \cto u$.
The second morphism sends an $n$-cylinder $z : u \cto v$ to the pair of $n$-arrows
$(v, u)$.
\end{tparagr}

\begin{thm}[Lafont-Métayer-Worytkiewicz]\label{thm:gamma_cyl}
If $G$ is a strict \oo-groupoid, then the above factorization $G \to
\Gamma(G) \to G \times G$ is actually a factorization in the category of
strict \oo-groupoids and~$\Gamma(G)$ endowed with this factorization is a
path object of $G$ in the Brown-Golasi\'nski model category structure.
\end{thm}

\begin{proof}
See Theorem 4.21 and Proposition 4.45 of \cite{LMW}.
\end{proof}

\begin{rem}
In this paper, we will not use the fact that $\Gamma(G) \to G \times G$ is a
fibration. The version of the above theorem we will need can thus be stated
only in terms of weak equivalences of strict \oo-groupoids (and in
particular, without reference to the Brown-Golasi\'nski model category
structure).
\end{rem}

\section{Disks are contractible}\label{sec:disks}

Throughout this section, we fix an integer $n \ge 0$. We will denote by $u$
the $n$-arrow of~$\Dnt{n}$ coming from the unique $n$-arrow of the globular
set $\Dn{n}$. The goal of the section is to prove that the \oo-groupoid
$\Dnt{n}$ is weakly contractible. For this purpose, we will exhibit an
$n$-cylinder from $u$ to (an iterated identity of) an object of $\Dnt{n}$.

\begin{tparagr}{A cylinder}
For $i$ such that $1 \le i \le n+1$, we define two $i$-arrows
$c^\flat_i$ and $c^\sharp_i$ of $\Dnt{n}$ by
\[
\begin{split}
c^\flat_i & = \Glid{u^\flat_0}, \\
c^\sharp_i & = {(u^\flat_1)}^{-1} \comp_0 \big({(u^\flat_2)}^{-1}
\comp_1 \big(\cdots \comp_{i-3} \big({(u^\flat_{i-1})}^{-1}
\comp_{i-2} {(u^\flat_i)}^{-1}\big)\big)\big).
\end{split}
\]
When $i = 1$, the definition of $c^\sharp_i$ should be read as $c^\sharp_1 =
{(u^\flat_1)}^{-1}$. Note also that the $(n+1)$\nbd-morphism $u^\flat_{n+1}$
is by definition equal to $\Glid{u}$. We leave to the reader the
verification that the formula defining $c^\sharp_i$ makes sense. By
definition, for $1 \le i \le n$, we have
\[
\Glt{i}(c^\sharp_i) = {(u^\flat_1)}^{-1} \comp_0 \big({(u^\flat_2)}^{-1}
\comp_1 \big(\cdots \comp_{i-3} \big({(u^\flat_{i-1})}^{-1} \comp_{i-2}
u^\flat_{i-1}\big)\big)\big).
\]
We claim that this formula simplifies to
\[ 
\Glt{i}(c^\sharp_i) = \Glid{u^\flat_0}.
\]
For $i = n+1$, we claim that we even have
\[
c^\sharp_{n+1} = \Glid{u^\flat_0}.
\]
These two claims are precisely the content of the following easy lemma:
\end{tparagr}

\begin{lemma}\label{lemma:target_cyl}
For $i$ such that $1 \le i \le n+1$, we have
\[
{(u^\flat_1)}^{-1} \comp_0 \big(
{(u^\flat_2)}^{-1} \comp_1 \big(
\cdots \comp_{i-3} \big(
{(u^\flat_{i-1})}^{-1} \comp_{i-2} u^\flat_{i-1}
\big)\big)\big) = \Glid{u^\flat_0}.
\]
\end{lemma}

\begin{proof}
The result is obvious for $i = 1$. For $i \ge 2$, we have
\[
\begin{split}
\MoveEqLeft {(u^\flat_1)}^{-1} \comp_0 \big(
{(u^\flat_2)}^{-1} \comp_1 \big(
\cdots \comp_{i-4} \big(
{(u^\flat_{i-2})}^{-1} \comp_{i-3} \big(
{(u^\flat_{i-1})}^{-1} \comp_{i-2} u^\flat_{i-1}
\big)\big)\big)\big)\\
& =
{(u^\flat_1)}^{-1} \comp_0 \big(
{(u^\flat_2)}^{-1} \comp_1 \big(
\cdots \comp_{i-4} \big(
{(u^\flat_{i-2})}^{-1} \comp_{i-3} u^\flat_{i-2}
\big)\big)\big)\\
\end{split}
\]
and the result follows by induction.
\end{proof}

\begin{tparagr}{A cylinder (sequel)}
By the previous lemma, for $i$ such that $1 \le i \le n + 1$, the sources
and targets of the $c^\epsilon_i$'s are given by
\[
\begin{split}
c^\flat_i & : \Glid{u^\flat_0} \to \Glid{u^\flat_0}, \\
c^\sharp_i & : {(u^\flat_1)}^{-1} \comp_0 \big({(u^\flat_2)}^{-1} \comp_1
\big(\cdots \comp_{i-3} \big({(u^\flat_{i-1})}^{-1} \comp_{i-2}
{(u^\sharp_{i-1})}^{-1}\big)\big)\big)
\to \Glid{u^\flat_0}.
\end{split}
\]
Moreover, we have
\[ c^\flat_{n+1} = \id{u^\flat_0} = c^\sharp_{n+1}. \]
This morphism will now simply be denoted by $c_{n+1}$.

Here are the corresponding diagrams for $n = 0, 1, 2$:
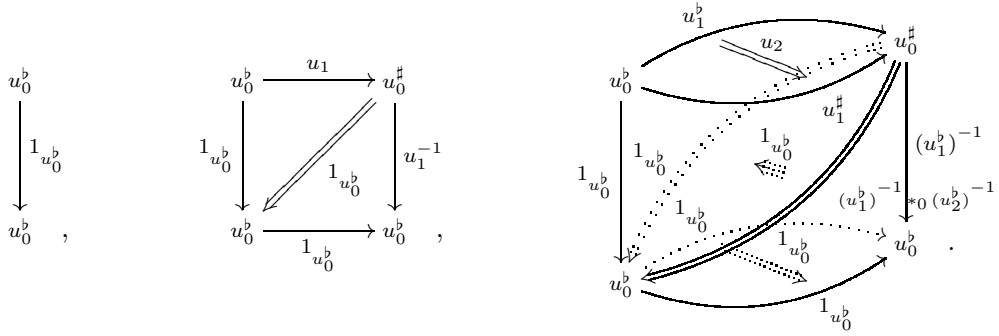
\[
\newcommand\st\scriptstyle
\xygraph{
!{<0cm,0cm>;<1cm,0cm>:<0cm,1cm>::}
!{(0,0) }*+{\st u^\flat_0}="u"
!{(0,-2) }*+{\st u^\flat_0\pbox{,}}="v"
"u":"v"^{\Glid{u^\flat_0}}
}
\qquad
\qquad
\xygraph{
!{<0cm,0cm>;<1cm,0cm>:<0cm,1cm>::}
!{(3,0) }*+{\st u^\flat_0}="ufz"
!{(5,0)}*+{\st u^\sharp_0}="usz"
!{(3,-2) }*+{\st u^\flat_0}="vfz"
!{(5,-2)}*+{\st u^\flat_0\pbox{,}}="vsz"
"ufz":"vfz"_{\Glid{u^\flat_0}}
"usz":"vsz"^{u_1^{-1}}
"ufz":"usz"^{u_1}
"vfz":"vsz"_{\Glid{u^\flat_0}}
"usz":@{=>}"vfz"^{\Glid{u^\flat_0}}
}
\qquad
\qquad
\xygraph{ 
!{<0cm,0cm>;<1.5cm,0cm>:<0cm,0.9cm>::} 
!{(0,0) }*+{\st u^\flat_0}="ufz"
!{(0,-3) }*+{\st u^\flat_0}="vfz"
!{(2.5,0.6)}*+{\st u^\sharp_0}="usz"
!{(2.5,-2.4)}*+{\st u^\flat_0\pbox{.}}="vsz"
!{(.8,0.6)}*+{}="ufo"
!{(1.7,0)}*+{}="uso"
!{(.8,-2.4)}*+{}="vfo"
!{(1.7,-3)}*+{}="vso"
!{(1.1,-1.2)}*+{}="codz"
!{(1.5,-1.4)}*+{}="domz"
"ufz":@/^.5cm/"usz"^(.3){u^\flat_1}
"ufz":@/_.55cm/"usz"_(.7){u^\sharp_1}
"ufo":@{=>}"uso"^{u_2}
"vfz":@{.>}@/^.5cm/"vsz"^(.3){\Glid{u^\flat_0}}
"vfz":@/_.55cm/"vsz"_(.7){\Glid{u^\flat_0}}
"vfo":@{:>}"vso"^(.7){\Glid{u^\flat_0}}
"ufz":"vfz"_{\Glid{u^\flat_0}}
"usz":"vsz"^{{(u^\flat_1)}^{-1}}
"usz":@{=>}@/^.7cm/"vfz"^(.4){\,\scriptscriptstyle {(u^\flat_1)}^{-1} \comp_0 {(u^\flat_2)}^{-1}}
"usz":@{:>}@/_.7cm/ "vfz"_(.7){\Glid{u^\flat_0}}
"domz":@3{.>}"codz"_{\Glid{u^\flat_0}}
}
\]
\end{tparagr}

\begin{prop}\label{prop:ext_cyl}
The arrows $c^\flat_0, c^\sharp_0, \ldots, c^\flat_n, c^\sharp_n, c_{n+1}$
define an $n$-cylinder $c : u \cto \Glid{u^\flat_0}$.
\end{prop}

\begin{proof}
We have to check that for $1 \le i \le n+1$ and $\epsilon = \flat,
\sharp$, we have
\[
\begin{split}
\Gls{i}(c^\epsilon_i) & = c^\sharp_{i-1} \comp_{i-2}
\big(c^\sharp_{i-2} \comp_{i-3} \big(\cdots \comp_1
\big(c^\sharp_1 \comp_0 u^\epsilon_{i-1}\big)\big)\big),
\\
\Glt{i}(c^\epsilon_i) & = \big(\big(\big(\Glid{u^\flat_0} \comp_0
c^\flat_1\big)\comp_1 c^\flat_2\big) \comp_2 \cdots\big)
\comp_{i-2} c^\flat_{i-1}.
\end{split}
\]
The second equality expands to
\[
\Glid{u^\flat_0} =
\Glid{u^\flat_0} \comp^{i-1}_0 \big(
\Glid{u^\flat_0} \comp^{i-1}_1 \big(
\comp_2 \cdots \big(
\Glid{u^\flat_0} \comp^{i-1}_{i-2} \Glid{u^\flat_0}
\big)\big)\big),
\]
which is obviously true. The first one is exactly the content (modulo Lemma
\ref{lemma:target_cyl} for the case $\epsilon = \flat$)
of the following lemma applied to $j = i - 1$:
\end{proof}

\begin{lemma}\label{lemma:comp_cyl}
For $i$ and $j$ such that $0 \le j < i \le n + 1$ and $\epsilon = \flat,
\sharp$, we have
\[
\begin{split}
\MoveEqLeft 
c^\sharp_j \comp_{j-1} \big(
c^\sharp_{j-1} \comp_{j-2} \big(
\cdots
\comp_1 \big(
c^\sharp_1 \comp_0 u^\epsilon_{i-1}
\big)\big)\big)\\
& =
{(u^\flat_1)}^{-1} \comp_0 \big(
{(u^\flat_2)}^{-1} \comp_1 \big(
\cdots
\comp_{j-2} \big(
{(u^\flat_j)}^{-1} \comp_{j-1} u^\epsilon_{i-1}
\big)\big)\big).
\end{split}
\]
\end{lemma}

\begin{proof}
Fix $i$ such that $1 \le i \le n + 1$. We prove the result by induction on
$j$. For $j = 0$, the result is obvious.  For $j \ge 1$, we have
\begin{align*}
\MoveEqLeft
c^\sharp_j \comp_{j-1} \big(
c^\sharp_{j-1} \comp_{j-2} \big(
\cdots
\comp_1 \big(
c^\sharp_1 \comp_0 u^\epsilon_{i-1}
\big)\big)\big)\\
& =
c^\sharp_j \comp_{j-1} \big(
{(u^\flat_1)}^{-1} \comp_0 \big(
{(u^\flat_2)}^{-1} \comp_1 \big(
\cdots
\comp_{j-3} \big(
{(u^\flat_{j-1})}^{-1} \comp_{j-2} u^\epsilon_{i-1}
\big)\big)\big)\big)\\
& =
\big(
{(u^\flat_1)}^{-1} \comp_0 \big({(u^\flat_2)}^{-1}
\comp_1 \big(\cdots \comp_{j-3} \big({(u^\flat_{j-1})}^{-1}
\comp_{j-2} {(u^\flat_j)}^{-1}\big)\big)\big)\big)
\comp_{j-1} \\*
& \qquad\quad
\big(
{(u^\flat_1)}^{-1} \comp_0 \big(
{(u^\flat_2)}^{-1} \comp_1 \big(
\cdots
\comp_{j-3} \big(
{(u^\flat_{j-1})}^{-1} \comp_{j-2} u^\epsilon_{i-1}
\big)\big)\big)\big),
\end{align*}
where the first equality is obtained by induction and the second one by
expanding the definition of $c^\sharp_j$. The result is then obtained by
applying $j-1$ times the exchange law:
\begin{align*}
\MoveEqLeft
\big(
{(u^\flat_1)}^{-1} \comp_0 \big({(u^\flat_2)}^{-1}
\comp_1 \big(\cdots \comp_{j-3} \big({(u^\flat_{j-1})}^{-1}
\comp_{j-2} {(u^\flat_j)}^{-1}\big)\big)\big)\big)
\comp_{j-1} \\*
& \qquad\quad
\big(
{(u^\flat_1)}^{-1} \comp_0 \big(
{(u^\flat_2)}^{-1} \comp_1 \big(
\cdots
\comp_{j-3} \big(
{(u^\flat_{j-1})}^{-1} \comp_{j-2} u^\epsilon_{i-1}
\big)\big)\big)\big)\\
& =
{(u^\flat_1)}^{-1} \comp_0 
\Big[
\big({(u^\flat_2)}^{-1}
\comp_1 \big(\cdots \comp_{j-3} \big({(u^\flat_{j-1})}^{-1}
\comp_{j-2} {(u^\flat_j)}^{-1}\big)\big)\big)
\comp_{j-1} \\*
& \qquad\qquad\qquad\quad
\big(
{(u^\flat_2)}^{-1} \comp_1 \big(
\cdots
\comp_{j-3} \big(
{(u^\flat_{j-1})}^{-1} \comp_{j-2} u^\epsilon_{i-1}
\big)\big)\big)\Big]\\
& =
{(u^\flat_1)}^{-1} \comp_0 
{(u^\flat_2)}^{-1} \comp_1
\Big[
\big({(u^\flat_3)}^{-1}
\comp_2 \big(\cdots \comp_{j-3} \big({(u^\flat_{j-1})}^{-1}
\comp_{j-2} {(u^\flat_j)}^{-1}\big)\big)\big)
\comp_{j-1} \\*
& \qquad\qquad\qquad\qquad\qquad\quad
\big(
{(u^\flat_3)}^{-1} \comp_2 \big(
\cdots
\comp_{j-3} \big(
{(u^\flat_{j-1})}^{-1} \comp_{j-2} u^\epsilon_{i-1}
\big)\big)\big)\Big]\\
& = \cdots\\
& =
{(u^\flat_1)}^{-1} \comp_0 \big(
{(u^\flat_2)}^{-1} \comp_1 \big(
\cdots
\comp_{j-2} \big(
{(u^\flat_j)}^{-1} \comp_{j-1} u^\epsilon_{i-1}
\big)\big)\big).
\end{align*}
\end{proof}

\begin{thm}\label{thm:Dnt_contr}
The strict \oo-groupoid $\Dnt{n}$ is weakly contractible.
\end{thm}

\begin{proof}
If suffices to show that the identify morphism $\Dnt{n} \to \Dnt{n}$ is
right homotopic to some constant morphism $\Dnt{n} \to \Dnt{n}$.
We will denote by $u^\flat_0$ the constant morphism $\Dnt{n} \to \Dnt{n}$
corresponding to the object $u^\flat_0$ of $\Dnt{n}$. By Theorem
\ref{thm:gamma_cyl}, it suffices to define a morphism $k : \Dnt{n} \to
\Gamma(\Dnt{n})$ making the diagram
\[
\xymatrix@R=1.5pc@C=3pc{
& & \Dnt{n} \\
\Dnt{n} \ar@/^/[urr]^{\id{\Dnt{n}}} \ar@/_/[drr]_{u^\flat_0} \ar[r]^k & \Gamma(\Dnt{n}) \ar[ru]_{p_0}
\ar[rd]^{p_1} \\
& & \Dnt{n}
}
\]
commute. But the data of such a morphism $k$ is clearly equivalent to the
data of an $n$-cylinder from $u$ to $\Glid{u^\flat_0}$. The result thus
follows from Proposition \ref{prop:ext_cyl}.
\end{proof}

\begin{rem}\label{rem:str_sat}
Let $\C$ be a category and let $\W$ be a class of morphisms of $\C$. Denote
by $p : C \to C[\W^{-1}]$ the localization functor. Recall that $\W$ is said
to be \ndef{strongly saturated} if every morphism $f$ of $C$ such that
$p(f)$ is an isomorphism is in $\W$. A sufficient condition for $\W$ to be
strongly saturated is that the pair $(\C, \W)$ is \ndef{Quillenisable},
i.e., that there exists a model category structure on $\C$ whose weak
equivalences are the elements of $\W$.

The proof that $\Dnt{n}$ is weakly contractible can be written so that the
Brown-Golasi\'nski model category structure is only used to prove that the
class of weak equivalences of strict \oo-groupoids is strongly saturated. In
particular, its cofibrations and fibrations play no role in this proof.
\end{rem}

\section{The globular extension $\Thtld$ is canonically
contractible}\label{sec:thtld_can_contr}

\begin{thm}\label{thm:LGT_contr}
Let $T$ be an object of $\Thtld$. Then the strict \oo-groupoid $L(G_T)$ is
weakly contractible.
\end{thm}

\begin{proof}
In the language of \cite{AraHomGr}, the fact that the $\Snt{n-1} \to \Dnt{n}$
are cofibrations and that the~$\Dnt{n}$'s are weakly contractible (Theorem
\ref{thm:Dnt_contr}) precisely means that the functor 
\[ \G \to \Thtld \to \wgpds \]
is cofibrant and weakly contractible. The result then follows from Proposition
5.11 of loc.~cit.

Let us briefly explain how to prove the result directly. Denote by $k$ the
width of $T$. If~$k = 1$, then $L(G_T) = \Dnt{i}$ for some $i$ and the
result follows from Theorem \ref{thm:Dnt_contr}. If~$k > 1$, then we have $T
= S \amalgd{i'} \Dn{i}$ for some $i, i'$ and some table of dimensions $S$ of
width $k - 1$. It follows that we have a cocartesian square of strict
\oo-groupoids
\[
\xymatrix{
\Dnt{i'} \ar[r] \ar[d]_{\Thtt[i']{i}} & L(G_S)  \ar[d] \\
  \Dnt{i}  \ar[r] & L(G_T) \pbox{.}\\
}
\]
By Remark \ref{rem:Dnt_cof}, the left vertical morphism is a cofibration. Since
$\Dnt{i'}$ and $\Dnt{i}$ are both weakly contractible, this morphism is
actually a trivial cofibration. It follows that the morphism $L(G_S) \to
L(G_T)$ is a trivial cofibration and the result follows by induction on~$k$.
\end{proof}

\begin{rem}
The above proof is our first real use of the Brown-Golasi\'nski model
category structure (see Remark \ref{rem:str_sat}).
\end{rem}

\begin{rem}
Theorem \ref{thm:LGT_contr} can be reformulated by saying that free strict
\oo-groupoids on a globular pasting scheme are weakly contractible.
\end{rem}

\begin{thm}\label{thm:thtld_can_contr}
The globular extension $\Thtld$ is canonically contractible.
\end{thm}

\begin{proof}
As we saw in the previous proof, in the language of \cite{AraHomGr}, the
functor $\G \to \wgpds$ is cofibrant and weakly contractible. Since every
strict \oo-groupoid is fibrant in the Brown-Golasi\'nski model category
structure, Proposition 5.12 of loc.~cit.~shows that the globular extension
$\Thtld$ is contractible.

Let us briefly explain how to prove the contractibility directly. Let $(f,
g) : \Dn{n} \to T$ be an admissible pair of $\Thtld$.  Explicitly, $f$ and
$g$ are morphisms of strict \oo-groupoids from~$\Dnt{n}$ to~$L(G_T)$.  Since
$f$ and $g$ are globularly parallel, they can be glued along $\Snt{n-1}$ and
hence define a morphism
\[ (f, g) : \Snt{n} = \Dnt{n} \amalg_{\Snt{n-1}} \Dnt{n} \to L(G_T). \]
One easily checks that a lifting of the pair $(f, g)$ is exactly a morphism
$h : \Dnt{n+1} \to L(G_T)$ making the triangle
\[
\xymatrix{
\Dnt{n+1} \ar[rd]^h \\
\Snt{n} \ar[u] \ar[r]_-{(f, g)} & L(G_T)
}
\]
commute or, in other words, a diagonal filler of the square
\[
\xymatrix{
\Snt{n} \ar[d] \ar[r] & L(G_T) \ar[d] \\
\Dnt{n+1} \ar[r] & \ast \pbox{,}
}
\]
where $\ast$ denotes the terminal strict \oo-groupoid. But such a diagonal
filler exists since the left vertical morphism is a cofibration and the
right vertical morphism is a trivial fibration by the previous theorem and
the fact that every strict \oo-groupoid is fibrant in the Brown-Golasi\'nski
model category structure.

Let us now prove that the globular extension $\Thtld$ is \emph{canonically}
contractible. We have to prove that the lifting $h$ constructed above is
unique. Recall that if $G$ is a strict \oo-groupoid and $k \ge 0$, there is a
canonical bijection between the morphisms \hbox{$\Dnt{k} \to G$} and the $k$-arrows of
$G$. In this proof, we will identify these morphisms with their
corresponding $k$-arrow. Using this identification, a lifting $h$ of the
pair $(f, g)$ is nothing but an $(n+1)$\nbd-arrow $h : f \to g$ of $L(G_T)$.
Let $h' : f \to g$ be a second $(n+1)$-arrow of $L(G_T)$. By the previous
theorem, the \oo-groupoid $L(G_T)$ is weakly contractible. It follows from
Proposition~\ref{prop:def_contr} that there exists an $(n+2)$-arrow $h \to
h'$ in $L(G_T)$. But by definition of an admissible pair, the dimension of
$T$ is at most $n+1$. It follows that $L(G_T)$ has no non-trivial
$(n+2)$-arrows. This shows that $h = h'$, thereby proving the result.
\end{proof}

\begin{rem}\label{rem:thtld_can_contr}
In the above proof, we used our additional condition in the definition of
admissible pairs (see Remark \ref{rem:def_adm}) in an essential way. If we
drop this condition, then the globular extension $\Thtld$ is no longer
\emph{canonically} contractible. Indeed, without this condition, the pair
\[ (\Thst[0]{2}, \Thtt[0]{2}) : \Dnt{0} \to \Dnt{2} \]
would be admissible. Nevertheless,  it has two different liftings, namely
$\Thst{2}$ and~$\Thtt{2}$.
\end{rem}

\section{Fully faithful functors between globular presheaf
categories}\label{sec:ff}

\begin{tparagr}{Precategorical globular extensions}
Let $C$ be a globular extension. A \ndef{precategorical structure}
on $C$ consists of the structure of a co-\oo-precategory on the
co-\oo-graph defined by the functor $\G \to C$. More explicitly, such a
structure is given by morphisms
\[
\begin{split}
\Thn[j]{i} & : \Dn{i} \to \Dn{i} \amalgd{j} \Dn{i},\qquad i > j \ge 0, \\
\Thk{i} & : \Dn{i+1} \to \Dn{i},\qquad i \ge 0,\\
\end{split}
\]
such that
\begin{itemize}
\item for every $i > j \ge 0$, we have
\[ 
\Thn[j]{i}\Ths{i} =
\begin{cases}
    \epsilon^{}_2\Ths{i}, & j = i - 1, \\
    \big(\Ths{i} \amalgd{j} \Ths{i}\big)\Thn[j]{i-1}
    & j < i - 1,
\end{cases}
\]
and
\[
\Thn[j]{i}\Tht{i} =
\begin{cases}
    \epsilon^{}_1\Tht{i}, & j = i - 1, \\
    \big(\Tht{i} \amalgd{j} \Tht{i}\big)\Thn[j]{i-1}
    & j < i - 1,
\end{cases}
\]
where $\epsilon^{}_1,\epsilon^{}_2 : \Dn{i} \to \Dn{i}\amalgd{i-1} \Dn{i}$
denote the canonical morphisms;
\item for every $i \ge 0$, we have
\[ \Thk{i}\Ths{i+1} = \id{\Dn{i}} \quad\text{and}\quad \Thk{i}\Tht{i+1} =
\id{\Dn{i}}\text{.} \]
\end{itemize}
A \ndef{precategorical globular extension} is a globular extension endowed
with a precategorical structure.
\end{tparagr}

\begin{prop}\label{prop:str_precat_ext}
Any contractible globular extension can be endowed with a precategorical
structure. Moreover, if the globular extension is canonically contractible,
then this structure is canonical.
\end{prop}

\begin{proof}
Let $C$ be a globular extension. The choice of a precategorical structure on
$C$ is equivalent to the choice of some liftings  of admissible pairs of $C$
(see Section 3 of \cite{AraHomGr} for details). In particular, if $C$ is
contractible, such a choice can always be made (see paragraph 3.11 of
loc.~cit.) and if $C$ is canonically contractible, this choice is unique.
\end{proof}

\begin{paragr}
Let $C$ be a precategorical globular theory. Then for any globular presheaf
$X$ on~$C$, the underlying globular set of $X$ (which is obtained by
pre-composing by $\G \to C$) is canonically endowed with the structure of
an \oo-precategory. This defines a functor
\[ \Mod{C} \to \wpcat. \]
This functor is easily seen to be faithful.
\end{paragr}

\begin{rem}
This functor can also be described in the following way. Let
$\Theta_{\textrm{pcat}}$ be the universal precategorical globular theory
(its existence can be shown using the globular completion, see paragraph 3.9
of \cite{AraThtld}). By definition of~$\Theta_{\textrm{pcat}}$,
a precategorical structure on a globular theory $C$ defines a globular
functor $\Theta_{\textrm{pcat}} \to C$. This functor is bijective on objects
and thus induces a faithful functor
\[ \Mod{C} \to \Mod{\Theta_{\textrm{pcat}}} \cong \wpcat. \]
\end{rem}

\begin{prop}
Let $C$ be a precategorical globular theory. For any globular functor from
$C$ to $\Thtld$, the induced functor
\[
\wgpds \cong \Mod{\Thtld} \to \Mod{C}
\]
is fully faithful.
\end{prop}

\begin{proof}
By Proposition \ref{prop:str_precat_ext} and
Theorem~\ref{thm:thtld_can_contr}, the globular extension $\Thtld$ is
endowed with a canonical precategorical structure. It follows that the
functor $C \to \Thtld$ respects the precategorical structures. This implies
that the triangle
\[
\xymatrix@C=.5pc{
\Mod{\Thtld} \ar[dr] \ar[rr] & & \Mod{C} \ar[dl] \\
& \wpcat \pbox{,}
}
\]
where the two oblique arrows are given by the respective precategorical
structures of~$\Thtld$ and $C$, is commutative. Moreover, the
functor $\wgpds \cong \Mod{\Thtld} \to \wpcat$ is nothing but the forgetful
functor and is hence fully faithful. The result then follows from the fact
that the functor $\Mod{C} \to \wpcat$ is faithful.
\end{proof}

\begin{coro}\label{coro:abs_ff}
Let $C$ be a contractible globular theory. For any globular functor from~$C$
to $\Thtld$, the induced functor
\[
\wgpds \cong \Mod{\Thtld} \to \Mod{C}
\]
is fully faithful.
\end{coro}

\begin{proof}
By Proposition \ref{prop:str_precat_ext}, $C$ can be endowed with a
precategorical structure. The result then follows from the previous
proposition.
\end{proof}

\section{Strict $\infty$-groupoids are Grothendieck $\infty$-groupoids}

\begin{thm}\label{thm:strweak}
Let $C$ be a coherator. There exists a canonical functor
\[ \wgpds \to \wgpdC{C}. \]
Moreover, this functor is fully faithful.
\end{thm}

\begin{proof}
By Theorem \ref{thm:thtld_can_contr}, the globular theory $\Thtld$ is
canonically contractible. It follows from
Proposition~\ref{prop:coh_weak_init} that there exists a unique globular
functor $C \to \Thtld$. This globular functor induces a functor
\[ \wgpds \cong \Mod{\Thtld} \to \Mod{C} = \wgpdC{C}, \]
thereby proving the first assertion. The second assertion follows
immediately from Corollary~\ref{coro:abs_ff}.
\end{proof}

\bibliographystyle{amsplain}
\bibliography{biblio}

\end{document}